\documentclass[12pt]{amsart}
\usepackage{amsmath,amssymb,amsfonts,amsthm,amstext,graphicx,color,enumerate}
\usepackage{hyperref,cleveref,url,mathtools}
\usepackage{tikz}
\usepackage[numbers,sort&compress]{natbib}

\title{One-dependent coloring by finitary factors}

\newtheorem{thm}{Theorem}
\newtheorem*{thm*}{Theorem}
\newtheorem{prop}[thm]{Proposition}
\newtheorem{lemma}[thm]{Lemma}
\newtheorem{cor}[thm]{Corollary}

\crefname{thm}{Theorem}{Theorems}
\crefname{lemma}{Lemma}{Lemmas}
\crefname{prop}{Proposition}{Propositions}
\crefname{cor}{Corollary}{Corollaries}
\crefname{section}{Section}{Sections}
\crefname{figure}{Figure}{Figures}

\newcommand{\E}{\mathbb E}
\renewcommand{\P}{\mathbb P}

\newcommand{\Z}{\mathbb Z}
\newcommand{\R}{\mathbb R}
\newcommand{\F}{\mathcal F}

\DeclareMathOperator{\id}{id}

\renewcommand{\hat}{\widehat}

\newcommand{\eqd}{\stackrel d=}

\renewcommand{\L}{\lhd}
\newcommand{\G}{\rhd}
\renewcommand{\l}{\triangleleft}
\newcommand{\g}{\triangleright}

\newcounter{mycount}
\newenvironment{ilist}{\begin{list}{\rm(\roman{mycount})}%
   {\usecounter{mycount}\labelwidth=4mm\itemsep 3pt\leftmargin=10mm}}{\end{list}}

\newcommand{\df}[1]{\textbf{\boldmath #1}}

\author[Alexander E. Holroyd]{Alexander E.\ Holroyd}
\address{Alexander E.\ Holroyd, Microsoft Research,
1 Microsoft Way, Redmond, WA 98052, USA} \email{holroyd at microsoft.com}
\urladdr{\url{http://research.microsoft.com/~holroyd/}}


\date{23 October 2014}

\begin{document}
\begin{abstract}
Holroyd and Liggett recently proved the existence of a
stationary $1$-dependent $4$-coloring of the integers,
the first stationary $k$-dependent $q$-coloring for any
$k$ and $q$.  That proof specifies a consistent family of
finite-dimensional distributions, but does not yield a
probabilistic construction on the whole integer line.
Here we prove that the process can be expressed as a
finitary factor of an i.i.d.\ process.  The factor is
described explicitly, and its coding radius obeys
power-law tail bounds.
\end{abstract}

\maketitle

\section{Introduction}

Let $X=(X_i)_{i\in\Z}$ be a stochastic process, i.e.\ a random element of
$\R^\Z$.  We call $X$ a (proper) \df{$q$-coloring} if each $X_i$ takes values
in $[q]:=\{1,\ldots,q\}$, and almost surely we have $X_i\neq X_{i+1}$ for all
$i\in\Z$.  A process $X$ is called \df{$k$-dependent} if the random vectors
$(X_i)_{i\in A}$ and $(X_i)_{i\in B}$ are independent of each other whenever
$A$ and $B$ are two subsets of $\Z$ satisfying $|a-b|>k$ for all $a\in A$ and
$b\in B$.  A process is \df{finitely dependent} if it is $k$-dependent for
some integer $k$.  A process $X$ is \df{stationary} if $(X_i)_{i\in\Z}$ and
$(X_{i+1})_{i\in\Z}$ are equal in law.  Holroyd and Liggett \cite{hl} proved
the existence of a stationary $1$-dependent $4$-coloring, and a stationary
$2$-dependent $3$-coloring.  These were the first known stationary finitely
dependent colorings.  The descriptions of the processes given in \cite{hl}
are mysterious, and involve specifying a consistent family of
finite-dimensional distributions rather than an explicit construction on
$\Z$.

A process $X$ is a \df{factor} of a process
$Y=(Y_i)_{i\in\Z}$ if it is equal in law to $F(Y)$, where
the function $F$ is translation-equivariant, i.e.\ it
commutes with the action of translations of $\Z$. A factor
of a stationary process is necessarily stationary. The
factor is \df{finitary} if for almost every $y$ (with
respect to the law of $Y$) there exists $t<\infty$ such
that whenever $y'$ agrees with $y$ on the interval
$\{-t,\ldots,t\}$, the resulting values assigned to $0$
agree, i.e.\ $F(y')_0=F(y)_0$.  In that case we write
$T(y)$ for the minimum such $t$, and we call the random
variable $T=T(Y)$ the \df{coding radius} of the factor. In
other words, in a finitary factor, the symbol $X_0$ at the
origin can be determined by examining only those variables
$Y_i$ within a random, finite, but perhaps unbounded
distance $T$ from the origin.

Our main result is that finitely dependent coloring can be done by a finitary
factor.
\begin{thm}\label{main}
There exists a $1$-dependent $4$-coloring of $\Z$ that is a finitary factor
of an i.i.d.\ process.  The coding radius $T$ satisfies
$$\P(T>t) < c \,t^{-\alpha},\qquad t\geq 1,$$
for some absolute constants $c,\alpha>0$.
\end{thm}
Our proof of \cref{main} gives an explicit description of the finitary factor
$F$, and our $4$-coloring $X$ is actually the same as the one in \cite{hl}
(but constructed in a different way). The power $\alpha$ that we obtain is
strictly less than $1$ (in fact, it is rather close to $0$), and $T$ has
infinite mean. We do not know whether there exists a finitely dependent
coloring that is a finitary factor of an i.i.d.\ process with finite mean
coding radius.

The result of \cite{hl} that stationary finitely dependent
colorings exist is surprising for several reasons.  A
\df{block factor} is a finitary factor with {\em bounded}
coding radius (so that $X_i$ is a fixed function of
$Y_{j-t},\ldots,Y_{j+t}$). Block factors of i.i.d.\
processes provide a natural means to construct finitely
dependent processes: if the coding radius is at most $k$
then the process is $2k$-dependent. Indeed, for some time
it was an open problem whether {\em every} finitely
dependent process was a block factor of an i.i.d.\ process
(see e.g.\ \cite{janson-84}). The first published
counterexample appears in
\cite{aaronson-gilat-keane-devalk}.  Prior to \cite{hl},
there was a (very reasonable) belief that most ``natural''
finitely dependent processes are block factors (see e.g.\
\cite{adding}). However, it turns out that no block-factor
coloring exists (see e.g.\ \cite{naor,hsw,hl}).  Hence, the
result of \cite{hl} shows that on the contrary, the very
natural task of {\em coloring} serves to distinguish
between block factors and finitely dependent processes. See
\cite{hl} for more on the history of this problem.

\subsection*{Extensions and applications}

We next discuss other $k$ and $q$. One can attempt to apply
our method to the $2$-dependent $3$-coloring of \cite{hl},
but we will see that it encounters a fundamental obstacle
in this case. By a result of Schramm (see \cite{hsw} or
\cite{hl}), no stationary $1$-dependent $3$-coloring
exists.  The stationary $1$-dependent $4$-coloring is
conjectured in \cite{hl} to be unique. (And $(k,q)=(1,4)$
is shown to be a critical point in a certain sense).  It is
plausible that the $2$-dependent $3$-coloring is unique
also, although the evidence is less strong in this case. It
is natural to expect that there is far more flexibility in
$k$-dependent $q$-colorings for larger $k$ and $q$,
although constructing examples seems difficult, and
currently very few are known. In \cite{hl2}, a
$1$-dependent $q$-coloring that is symmetric in the colors
is constructed for each $q\geq 4$. Besides these colorings
and the two in \cite{hl}, and straightforward
embellishments thereof, no other examples are known.  One
such embellishment, described in \cite{hl}, is a
$3$-dependent $3$-coloring that arises as a simple block
factor of the $1$-dependent $4$-coloring.  Since a
composition of finitary factors is finitary, \cref{main}
implies that this $3$-dependent $3$-coloring is a finitary
factor of an i.i.d.\ process also.

In light of the above observations, a natural conjecture is
that there exists a $k$-dependent $q$-coloring that is a
finitary factor of an i.i.d.\ process with finite mean
coding radius if and only if $k\geq 1$, $q\geq 3$, and
$(k,q)\not\in\{(1,3),(2,3),(1,4)\}$.

Coloring on $\Z$ is a key case within a more general
framework.  It is proved in \cite{hl} that for every shift
of finite type $S$ on $\Z$ that satisfies a certain
non-degeneracy condition, there is a stationary finitely
dependent process that lies a.s.\ in $S$. Additionally, in
the $d$-dimensional lattice $\Z^d$, there exist a
$1$-dependent $q$-coloring and a $k$-dependent $4$-coloring
(where $q$ and $k$ depend on $d$).  These facts are proved
by starting from the $1$-dependent $4$-coloring of $\Z$ and
applying block-factors (in some cases using methods
developed in \cite{hsw}). Consequently, using our
\cref{main}, each of these processes is a finitary factor
of an i.i.d.\ process.

The question of $q$-coloring as a finitary factor of a i.i.d.\ process
(without the finite dependence requirement) is addressed in \cite{hsw},
including on $\Z^d$.  Depending on $q$ and $d$, the best coding radius tail
that can be achieved is either a power law (when $q=3$ and $d\geq 2$) or a
tower function (when $q \geq 4$ and $d \geq 2$, or $q\geq 3$ and $d=1$).

Coloring has applications in computer science.  For
instance, colors may represent time schedules or
communication frequencies for machines in a network, and
adjacent machines are not permitted to conflict with each
other.  Finite dependence implies a security benefit
--- an adversary who gains knowledge of some colors learns nothing about the
others, except within a fixed finite distance.  A finitary
factor of an i.i.d.\ process is also desirable. It has the
interpretation that the colors can be computed by the
machines in {\em distributed} fashion, based on randomness
generated locally, combined with communication with
machines within a finite distance.  All machines follow the
same protocol, and no central authority is needed.  See
e.g.\ \cite{linial,naor} for more information.
Unfortunately, the finitary factor of \cref{main} is of
limited practical use because of the heavy tail of $T$ ---
the communication distance to determine the color at the
origin is almost surely finite, but typically huge.

\subsection*{Outline of Proof}
As mentioned earlier, the construction in \cite{hl}
involves specifying the law of the coloring restricted to a
finite interval, and proving that these laws form a
consistent family.  The law on an interval has a
probabilistic interpretation, involving inserting colors in
a random order.  However, this order is not uniform, but
weighted to favor insertions at the endpoints.  The random
orders themselves are therefore {\em not} consistent
between different intervals.  Using uniformly random orders
instead gives consistent orders but inconsistent colorings.
However, we show that, in the limit of a long interval, the
choice of weight does not affect the law near the center.
The key observation is that endpoint insertions are
typically few (in fact $\Theta(\log n)$ in an interval of
length $n$), even under the weighting, so their effect can
be neglected.

To obtain a factor of an i.i.d.\ process we introduce a ``graphical
representation'' of the insertion process that extends naturally to $\Z$.
Each inserted color must itself be randomly chosen, and must differ from its
neighbors.  With $4$ (or more) colors, this means that there is a choice, and
we use this to define special locations at which the random color can be decoupled
from its surroundings, leading to a finite coding radius.

\section{The construction}

Fix a number of colors $q\geq 3$.  Let $(U_i)_{i\in\Z}$ be
i.i.d.\ random variables, with each $U_i$ uniformly
distributed on the interval $[0,1]$.  We interpret $U_i$ as
the {\em arrival time} of the integer $i$.  The idea is
that when an integer arrives, it chooses a color uniformly
at random from the colors that are not present among its
two current neighbors (i.e.\ the nearest integers to the
left and to the right that arrived before it).  Thus, for
$i\in \Z$, define:
\begin{align*}
  L(i)&:=\max\{j<i: U_j < U_i\};\\
  R(i)&:=\min\{j>i: U_j < U_i\}.
\end{align*}
Let $(\phi_i)_{i\in\Z}$ be i.i.d.\ permutations, with each $\phi_i$ uniformly
distributed on the symmetric group $S_q$ of permutations of $[q]$, and with
$(\phi_i)_{i\in\Z}$ independent of $(U_i)_{i\in\Z}$.  The idea is that
$\phi_i$ denotes the {\em preference order} of integer $i$ over the $q$
colors, with $\phi_i(j)$ being $i$'s $j$th favorite color, for $j=1,\ldots,q$.

Given $(U_i)_{i\in\Z}$ and $(\phi_i)_{i\in\Z}$, we seek a sequence
$(X_i)_{i\in \Z}\in[q]^\Z$ that satisfies the system of equations
\begin{equation}\label{system}
\begin{split}
X_i&=\phi_i(K), \qquad i\in\Z,\\
\text{where}\quad K&=K(i)=\min\Bigl\{k\in[q]: X_{L(i)}\neq \phi_i(k)\neq
X_{R(i)}\Bigr\}.
\end{split}
\end{equation}
Thus, the integer $i$ is assigned its favorite color $X_i$ among those that
have not been taken by its current neighbors at its arrival time.

It is clear that any solution $(X_i)_{i\in \Z}$ to \eqref{system} is a
$q$-coloring: we have $X_{L(i)}\neq X_i\neq X_{R(i)}$, but if $U_{i+1}<U_i$
then $R(i)=i+1$, and otherwise $L(i+1)=i$.  However, it is not immediately
clear whether there is a solution: $X_i$ is expressed in terms of $X_{L(i)}$
and $X_{R(i)}$, so the computation apparently involves an infinite regress.

The outcome depends crucially on the number of colors. By a
similar argument to the above, any solution must have
$X_{L(i)}\neq X_{R(i)}$ for all $i$, so precisely two
colors are ruled out for $X_i$ by the requirement
$X_{L(i)}\neq X_i\neq X_{R(i)}$.  Therefore, if $q=3$ then
only one color remains, so the preference order $\phi_i$ is
irrelevant.  On the other hand, if $q\geq 4$ then $i$ has a
choice, therefore it will never need its $4$th favorite
color $\phi_i(4)$ or worse.  This will allow us to end the
regress: if the favorite of $i$ is the $4$th favorite of
$L(i)$ and $R(i)$, then we know that $i$ will receive its
favorite color. Using this idea, we will prove the
following in the next section.

\begin{prop}\label{finitary}
Fix $q\geq 4$.  Let $(U_i)_{i\in\Z}$ be i.i.d.\ uniform on $[0,1]$, and
$(\phi_i)_{i\in\Z}$ i.i.d.\ uniform on $S_q$, independent of each other.
Almost surely, the system of equations \eqref{system} has a unique solution
$(X_i)_{i\in \Z}$.  Moreover, $(X_i)_{i\in \Z}$ is a finitary factor of the
i.i.d.\ process $((U_i,\phi_i))_{i\in \Z}$, with coding radius $T$ satisfying
$$\P(T>t)<c\,t^{-\alpha},\qquad t>0,$$
for some $c,\alpha>0$ depending only on $q$.
\end{prop}

In the later sections we will prove that when $q=4$, the process
$(X_i)_{i\in\Z}$ coincides with the $1$-dependent $4$-coloring constructed in
\cite{hl}.

When $q=3$ it is not difficult to check that a.s.\ the system of equations
\eqref{system} has exactly $3!=6$ solutions.  There is a.s.\ a uniquely defined
partition of $\Z$ into $3$ sets, which depends on $(U_i)_{i\in Z}$ but not on
$(\phi_i)_{i\in\Z}$, and each solution corresponds to an assignment of colors
to the $3$ sets.  If this assignment chosen to be a uniformly random
permutation in $S_3$, independent of $(U_i)_{i\in Z}$, then the resulting
process is the $2$-dependent $3$-coloring of \cite{hl}. However, this
``global assignment'' step means that the construction is not a finitary
factor.

\section{Coding radius bound}

In this section we prove \cref{finitary}.  Motivated by the discussion of the
previous section, given $(U_i,\phi_i)_{i\in \Z}$, we say that the integer
$i$ is \df{lucky} if
$$\phi_{L(i)}(4)=\phi_i(1)=\phi_{R(i)}(4).$$
(We could allow $\phi_{L(i)}(k),\phi_{R(i)}(k')$ for any
$k,k'\geq 4$ here, but the above definition suffices, and
in any case our main focus is $q=4$). Here is the key step.

\begin{lemma}\label{lucky}
Let $q\geq 4$ and let $(U_i,\phi_i)_{i\in \Z}$ be as in \cref{finitary}.
For any $m\geq 0$, a.s.\ there exist lucky integers $A$ and $B$ with
$[-m,m]\subseteq[A,B]$ such that
\begin{equation}\label{ends}
\max\{U_A,U_B\}<\min\{U_i:i\in(A,B)\}.
\end{equation}
Moreover, $A$ and $B$ can be chosen so that
\begin{equation}\label{prob-max}
\P\Bigl(\max\bigl\{|L(A)|,|R(A)|,|L(B)|,|R(B)|\bigr\}>t\Bigr)<c\,t^{-\alpha},
 \quad t\geq 1,
\end{equation}
for some positive constants $c=c(q,m)$ and $\alpha=\alpha(q)>0$.
\end{lemma}

\begin{proof}[Proof of \cref{finitary}]
Using \cref{lucky} with $m=0$, let $T\geq 1$ be the smallest integer for which
there exist lucky $A\leq 0 \leq B$ satisfying \eqref{ends} for which the
maximum appearing in \eqref{prob-max} is at most $T$.  This $T$ can be
determined from the variables $(U_i,\phi_i)$ for $|i|\leq T$ (by examining
the integers $i$ in order of absolute value), and \cref{lucky} states that it
satisfies the sought tail bound.

Since $A$ and $B$ are lucky, in any solution to \eqref{system} we have that
$X_A=\phi_A(1)$ and $X_B=\phi_B(1)$.  Condition \eqref{ends} implies that for
all $i\in(A,B)$ we have $L(i),R(i)\in[A,B]$.  Therefore, the remaining colors
$(X_i)_{i\in(A,B)}$ in the interval can be determined via \eqref{system} from
$((U_i,\phi_i))_{i\in [A,B]}$, by considering them in increasing order of
$U_i$.  In particular, we can determine $X_0$.

By stationarity, we can similarly find an interval $[A_i,B_i]\ni i$
corresponding to any $i\in \Z$.  By applying \cref{lucky} with larger $m$, it
is easy to see that the colors $(X_j)_{j\in[A_i,B_i]}$ computed from
different intervals are consistent with each other and with \eqref{system}.
We conclude that the resulting $(X_i)_{i\in \Z}$ is the unique solution to
\eqref{system}, and is a finitary factor of $(U_i,\phi_i)_{i\in \Z}$ with
coding radius $T$.
\end{proof}

Before proving \cref{lucky}, we briefly discuss where the power law tail bound
comes from.  First note that even $R(0)$ has mean $\infty$, since it is the
location of the second record minimum of the i.i.d.\ sequence $(U_i)_{i\geq
0}$ (the first record being at $i=0$). However, the integer $B$ of
\cref{lucky} is much larger than this. In addition to being a record, it must
be lucky. Consider the simplified situation in which $(G_i)_{i\in\Z}$ are
independent events of probability $p$, independent of $(U_i)_{i\in\Z}$. Let
$J$ be the smallest positive integer for which $G_J$ occurs and $(U_i)_{i\geq
0}$ has a record minimum at $J$.  Then, by the standard fact (see e.g.\
\cite[Example 2.3.2]{durrett}) that there is a record at $i$ with probability
$1/i$, independently for different $i$, we have
$$\P(J>t)=\prod_{i=1}^t \Bigl(1-\frac p i\Bigr)=\Theta(t^{-p})
\quad\text{as }t\to\infty.$$

For $q=4$, the probability that an integer is lucky is $1/16$, therefore at
best we can expect a power $\alpha=1/16$ in our tail bound on the coding
radius. We have not attempted to optimize $\alpha$, so the bound we prove is
in fact much smaller than this.  At the expense of additional complexity, our
method below could be adapted to prove a power closer to $1/16$.  By refining
the definition of lucky integers to encompass more complicated local patterns
of preferences, the bound could likely be increased beyond $1/16$. However,
any such improvement would still result in a power $\alpha$ strictly less
than $1$, and infinite mean coding radius.

\begin{proof}[Proof of \cref{lucky}]
To find suitable $A$ and $B$ we examine the integers in order of absolute
value.  Call $i\in\Z$ an \df{absolute record} if $U_i$ is smaller than all
the terms that precede it in the sequence $U_{i_1},U_{i_2},U_{i_3},\ldots,$
where $(i_1,i_2,i_3,\ldots)=(0,1,-1,2,-2,\ldots)$ Since the reordered
sequence is of course still i.i.d., $i_j$ is an absolute record with
probability $1/j$, and the events that different integers are absolute
records are independent.

Now, for $n\geq 1$, we compute the probability
\begin{equation*}
\begin{split}
\begin{aligned}
\P\Bigl[&[2^n,2^{n+1}) \text { contains exactly one absolute record, and}\\
 &(-2^{n+1},-2^n] \text { contains no absolute records}\Bigr]
\end{aligned}\\
\begin{aligned}
&=\prod_{j=2^{n}}^{2^{n+1}-1} \Bigl(1-\frac{1}{2j}\Bigr)\Bigl(1-\frac{1}{2j+1}\Bigr)
\cdot\sum_{i=2^n}^{2^{n+1}-1} \frac{\tfrac 1{2i}}{1-\tfrac 1{2i}}\\
&\geq \frac{1}{8}
\end{aligned}
\end{split}
\end{equation*}
(The product telescopes, and the sum can be bounded via its smallest term).
Similarly, the probability that $(-2^{n+1},-2^n]$ contains exactly one
absolute record while $[2^n,2^{n+1})$ contains none is also at least $1/8$.

For $n\geq 1$, let $E_n$ be the event that there exist integers $a,b,c,d,e$
satisfying all of:
\begin{ilist}
  \item
  $2^{5n}{\leq a <} 2^{5n+1} {\leq b <} 2^{5n+2}
  {\leq -c <} 2^{5n+3}{\leq -d <}2^{5n+4}{\leq e <}2^{5n+5}$;
  \item
  $a,b,c,d,e$ are the only absolute records in
  \sloppypar $(-2^{5n+5},-2^{5n}]\cup[2^{5n},2^{5n+5})$;
  \item
  $\phi_c(4)=\phi_a(1)=\phi_b(4)$, and $\phi_d(4)=\phi_c(1)=\phi_e(4)$.
\end{ilist}
On $E_n$, we have $L(a)=c$, $R(a)=b$, $L(c)=d$, and $R(c)=e$; therefore, $a$
and $c$ are lucky; we take $A=c$ and $B=a$.  We have
$[-2^{5n},2^{5n}]\subseteq [A,B]$, and \eqref{ends} holds.  Moreover, the
maximum in \eqref{prob-max} is at most $2^{5n+5}$.

On the other hand, the events $(E_n)_{n\geq 1}$ are independent.  Using the
previous computation, we have $\P(E_n)\geq (1/8)^5 (1/q)^4$.  We conclude
that the claimed bound holds with $\alpha=-\log(1-8^{-5}q^{-4})/(5\log 2)$.
(When $q=4$, this is approximately $3\times 10^{-8}$).
\end{proof}

\section{Weighted insertion processes}

We introduce a family of random proper colorings of finite intervals, which
we call \df{weighted insertion} (WI) colorings. Throughout this section, an
interval $[a,b]$ is understood to denote the set of integers $[a,b]\cap\Z$,
where $a,b\in\Z$.  A finite sequence $x=(x_i)_{i\in[a,b]}$ is a \df{coloring}
of $[a,b]$ if $x_i\neq x_{i+1}$ for all $a\leq i <b$.

Fix a number of colors $q\geq 3$ and a real \df{weight} $w>0$.  For $n\geq
1$, we define the \df{WI coloring} $X=X^{[n]}\in[q]^n$ of the interval
$[n]=[1,n]$ (with parameters $(w,q)$), via an iterative constriction. When
$n=1$, $X^{[1]}$ is a sequence of length $1$ consisting of a uniformly random
color from $[q]$. Conditional on $X^{[n]}=(X_1,\ldots,X_n)$, we construct
$X^{[n+1]}$ by the following insertion procedure.

First, we choose a random
\df{insertion point} $I$, with law that is uniform on $[1,n+1]$
except that the two endpoints have bias $w$:
$$\P(I=i)=
\begin{cases}
\frac{w}{2w+n-1},&i=1 \text{ or }n+1;\\
\frac{1}{2w+n-1},&i=2,\ldots,n.
\end{cases}
$$
Then we choose a random color $Z$ uniformly from the set
$$[q]\setminus\{X_{I-1},X_I\}$$
of colors that differ from the neighbors of the insertion point.  Here, $X_0$
and $X_{n+1}$ are taken to be $\infty$ (say), so that the above set has size
$q-1$ if $I\in\{1,n+1\}$, and otherwise size $q-2$ (since $X_{I-1}\neq X_I$ in a
proper coloring).  Finally, we insert $Z$ just before location $I$ to form
$X^{[n+1]}$:
$$X^{[n+1]}:=(X_1,\ldots,X_{I-1},Z,X_I,\ldots,X_n).$$

For an arbitrary integer interval $[a,b]$, we define the WI coloring
$X^{[a,b]}=(X_a,\ldots,X_b)$ to be simply equal in distribution to
$X^{[n]}=(X_1,\ldots,X_n)$ where $n=b-1+1$.  (No particular joint law is
assumed between different intervals, at present).

For $x\in [q]^n$, let
$$P(x)=P^{q,w}_n(x):=\P(X^{[n]}=x)$$
denote the probability mass function of the WI coloring $X^{[n]}$ of length
$n$. The above iterative description immediately gives rise to a recurrence
for $P$. Let $\hat x_i:=(x_1,\ldots,x_{i-1},x_{i+1},\ldots,x_n)$ denote the
sequence $x$ with its $i$th element deleted.  Then, for $n\geq 2$, if $x\in
[q]^n$ is a proper coloring,
$$P(x)=\frac{1}{2w+n-2}\biggl[
\frac{w}{q-1} \bigl[P(\hat x_1)+P(\hat x_n)\bigr]+\frac{1}{q-2}\,{\sum_{i=2}^{n-1}} P(\hat x_i)
\biggr],
$$
and $P(x)=0$ if $x$ is not a proper coloring.

Two special choices of the weight $w$ play an important role.  The first
is
\begin{equation}\label{w-star}
w^*=w^*(q):=\frac{q-1}{q-2}.
\end{equation}
In this case, the mass function of the insertion point $I$ is proportional to
the number of possible colors that are available for insertion at that point
($q-1$ or $q-2$ according to whether or not it is an endpoint), so the
insertion procedure amounts to choosing the pair $(I,Z)$ uniformly from the
set of all its allowed values.  In this case, the above recurrence reduces to
\begin{equation}\label{recur}
P(x)=\frac{1}{n(q-2)+2} \sum_{i=1}^n P(\hat x_i).
\end{equation}
Moreover, we have the following.

\begin{prop}[Holroyd and Liggett \cite{hl}]\label{col-cons}
Let $q\geq 3$ and $w=w^*(q)$.  The laws of the WI colorings on integer
intervals are consistent.  That is, if $[a,b]\subseteq [A,B]$ then the
restriction $X^{[A,B]}|_{[a,b]}$ is equal in law to $X^{[a,b]}$.
\end{prop}

The proof of \cref{col-cons} in \cite{hl} is a fairly straightforward
induction using \eqref{recur}. By the Kolmogorov extension theorem,
\cref{col-cons} implies that there exists a stationary coloring $(X_i)_{i\in
\Z}$ on the infinite line whose restrictions to intervals are given by the WI
model. This process has the following surprising properties.  The proof given
in \cite{hl} is again by induction using \eqref{recur}, and is short but
mysterious.

\begin{thm}[Holroyd and Liggett \cite{hl}]\label{kdep}
The stationary coloring $(X_i)_{i\in \Z}$ that extends the WI model with
weight $w^*$ is $1$-dependent when $q=4$, and $2$-dependent when $q=3$, but
not finitely dependent when $q\notin\{3,4\}$.
\end{thm}

It is easy to check that the consistency property of \cref{col-cons} does not
hold for any choice of weight other than $w^*$.  The main purpose of this
section is to prove the following proposition, which states that $w^*$ is an
attracting fixed point under restriction.

\begin{prop}\label{limit-col}
Fix $q\geq 3$, $w>0$ and $m\geq 0$.  For $n\geq m$, let $X^n$ be the WI
coloring on $[-n,n]$ with parameters $(q,w)$.  As $n\to \infty$, the
restriction $X^n|_{[-m,m]}$ of the coloring to $[-m,m]$ converges in law to
the WI coloring on $[-m,m]$ with parameters $(q,w^*)$.
\end{prop}

The second important choice of weight is $w=1$.  To explain the significance
of this case, we first observe that there is a random total order of the
interval naturally associated to the WI coloring, which records the order in
which the color insertions took place.

To define the random order formally, it is convenient to work first on
$[n]=[1,n]$ and iteratively construct the equivalent permutation of $[n]$,
which will be denoted $\Pi^{[n]}$. Set $\Pi^{[1]}=(1)=\id\in S_1$, and, given
$\Pi^{(n)}=(\Pi_1,\ldots,\Pi_n)\in S_n$, let $\Pi^{(n+1)}$ be obtained by
inserting $n+1$ at the same location that the new color was inserted:
$$\Pi^{[n+1]}:=(\Pi_1,\ldots,\Pi_{I-1},n+1,\Pi_I,\ldots,\Pi_n).$$

The \df{WI order} on $[1,n]$ is the random total order $\L=\L^{[n]}$ defined
by $i\L j$ if and only if $\Pi^{[n]}_i <\Pi^{[n]}_j$.  The \df{WI model} with
parameters $(q,w)$ on $[1,n]$ specifies the joint law of the coloring
$X=X^{[n]}$ and the order $\L=\L^{[n]}$.  On an interval $[a,b]$ with
$b-a+1=n$, we similarly define the joint law of $(X^{[a,b]},\L^{[a,b]})$ by
setting $X^{[a,b]}=X^{[n]}$ (as before), and $i\L^{[a,b]}j$ if and only if
$i-a+1 \L^{[n]} j-a+1$.

On given interval, observe that the law of the order $\L$ depends on $w$ but
not on $q$, while the conditional law of the coloring $X$ given $\L$ depends
on $q$ but not on $w$.  We investigate both of these laws below. First we
note the following.

\begin{lemma}
When $w=1$, the WI order on $[a,b]$ is a uniformly random total order on
$[a,b]$.  (In particular, for $[a,b]\subseteq[A,B]$ we have the consistency
$\L^{[A,B]}|_{[a,b]}\eqd \L^{[a,b]}$).
\end{lemma}

\begin{proof}
This is immediate from the iterative description, since the insertion point
$I$ is uniformly distributed on $[1,n+1]$.
\end{proof}

It is easy to check that consistency of the order does not hold for any other
weight.  Since $1\neq w^*$, it is interesting that consistency cannot hold
simultaneously for both the coloring and the order.

We now consider the law of the WI order in more detail.  For any total order
$\l$ on an interval $[a,b]$, we define the set of \df{founders} to be
\begin{align*}
\F(\l):=\{i\in[a,b]: \quad  i\l j &\;\forall\; j<i, \\
 \text{or }i \l j &\;\forall\; j>i\}.
\end{align*}
If we regard $i$ as a point in the plane with horizontal coordinate $i$ and
vertical coordinate given by its position in the order $\l$, the founders are
those points whose lower-left or lower-right quadrant contains no other
points; see \cref{founders}.  (Also, $\F(\l)$ is the set of indices at which
the inverse of the associated permutation attains a historical minimum or
maximum). \setlength{\unitlength}{0.01\textwidth}
\begin{figure}\centering
  \begin{picture}(0,0)
    \put(-4,13){$\triangledown$}
    \put(0,0){\line(1,0){66}}
    \put(0,0){\line(0,1){28}}
  \end{picture}
  \includegraphics[width=0.65\textwidth]{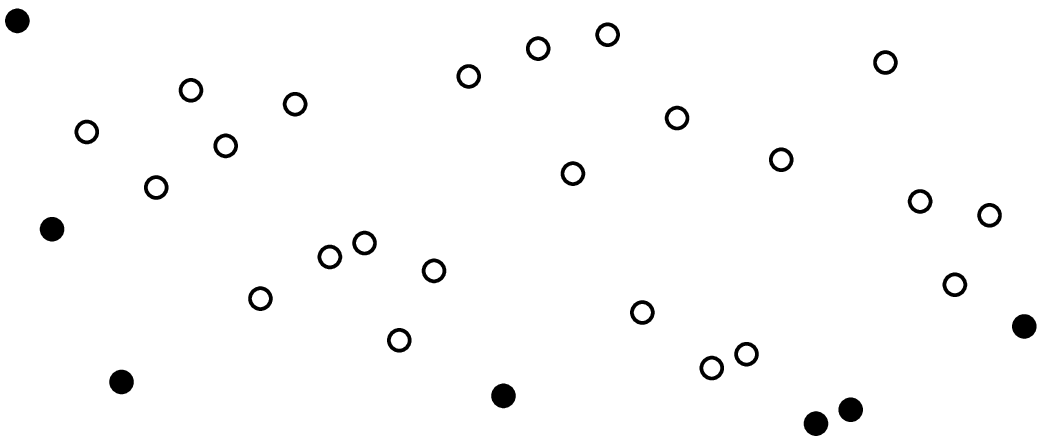}\\
  \begin{picture}(65,3)
    \put(0,0){$1$}
    \put(63,0){$30$}
  \end{picture}
\caption{A random total order $\L$ of the interval $[1,30]$,
with founders shown as filled discs, and other elements as unfilled discs.}
\label{founders}
\end{figure}

In the iterative description for the WI model on $[1,n]$, the founders are
the indices $i$ at which the color $X_i$ was inserted at an endpoint of the
interval during the relevant insertion step (including the case of the first
color to be chosen). An immediate consequence is that the law of the random
WI order $\L$ on $[1,n]$ is given by
\begin{equation}\label{partition}
\P(\L=\l)=\frac{w^{|\F(\l)|}}{Z(w,n)},
\end{equation}
where $\l$ is any of the $n!$ deterministic total orders on $[1,n]$, and
$Z(w,n)$ is an appropriate normalizing constant.

Our next goal is to prove the following analogue of \cref{limit-col} for WI
orders.  This time, $w=1$ is the attracting fixed point.

\begin{prop}\label{limit-unif}
Fix $w>0$ and $m\geq 1$. For $n>m$, let $\L^n$ be WI order on $[-n,n]$.  As
$n\to\infty$, the restriction $\L^n|_{[-m,m]}$ converges in law to a
uniformly random total order on $[-m,m]$.
\end{prop}

The following is the only estimate that we need in this section.

\begin{lemma}\label{log}
Fix $w>0$ and $n\geq 1$, and let $\L^{[n]}$ be WI order with weight $w$ on
$[1,n]$. If $n/3 <k <2n/3$ then
$$\P\bigl(k \in \F(\L^{[n]})\bigr) < \frac{c\log n}{n},$$
where $c$ is a constant depending only on $w$.
\end{lemma}

\begin{proof}
Fix $w$, and write $p_{n,k}:=\P(k \in \F(\L^{[n]}))$.  By symmetry we have
$p_{n,n+1-k}=p_{n,k}$, and $p_{n,1}=p_{n,n}=1$ since the endpoints of an
interval are always founders.  Considering the last step of the insertion
procedure gives the recurrence
$$p_{n+1,k}=\frac{1}{2w+n-1}\bigl[
(w+k-2)p_{n,k-1} + (w+n-k)p_{n,k} \bigr]$$ for $2\leq k \leq n$. (Here the
two terms on the right reflect the possibilities that the insertion is
before, of after, location $k$; if the insertion is \textit{at} $k$ then $k$
is not a founder). A straightforward induction then shows that $p_{n,k}$
is unimodal in $k$:
\begin{equation}\label{uni}
p_{n,k}\geq p_{n,k+1},\quad k<\tfrac n2.
\end{equation}

On the other hand, writing $s_n:=\sum_i p_{n,i}=\E
|\F(\L^{[n]})|$, we obtain similarly $s_1=1$ and
$$s_{n+1}=s_n+\frac{2w}{2w+n-1},$$
and hence
\begin{equation}\label{sum}
s_n \leq c_1 \log n
\end{equation}
for some $c_1=c_1(w)$.

Now if $n/3<k<2n/3$ then \eqref{sum} and \eqref{uni} give
$$c_1 \log n \geq s_n \geq
\sum_{i=1}^{\lceil n/3\rceil} p_{n,i} \geq \tfrac n3\,
p_{n, \lceil n/3\rceil},$$
 so $p_{n, \lceil n/3\rceil}\leq (3c_1 \log n )/n$, and
the result follows from \eqref{uni}.
\end{proof}

\begin{cor}\label{no-found}
Fix $w>0$ and $m\geq 1$.  For $n\geq m$, let $\L^{n}$ be the WI order with
weight $w$ on $[-n,n]$.  We have
$$\P\bigl(\F(\L^{n})\cap[-m,m]=\emptyset)\to 1 \qquad\text{as }n\to\infty.$$
\end{cor}

\begin{proof}
  This follows from \cref{log} by a union bound.
\end{proof}

\begin{lemma}\label{cond-unif}
Fix $w>0$, and integer intervals $[a,b]\subseteq[A,B]$.  Let $\L$ be the WI
order on $[A,B]$.  Given the event $\F(\L)\cap[a,b]=\emptyset$, the
conditional law of the restriction $\L|_{[a,b]}$ is the uniform measure on
total orders of $[a,b]$.
\end{lemma}

\begin{proof}
Consider the representation \eqref{partition} of the law of $\L$.  Let
$\l_1,\l_2$ be two total orders on $[a,b]$ that differ by a single transposition.
The two events
$$\bigl\{\F(\L)\cap[a,b]=\emptyset,\;
\L|_{[a,b]}=\l_i\bigr\},\qquad i=1,2$$ correspond to two (disjoint) sets of
total orders on $[A,B]$.  There is an explicit bijection between these sets
that preserves the set of founders: we simply exchange the positions of the
two transposed elements of $[a,b]$ within the order on $[A,B]$.
\end{proof}

\begin{proof}[Proof of \cref{limit-unif}]
This is an immediate consequence of \cref{no-found,cond-unif}.
\end{proof}

We now shift our focus to the conditional law of the WI coloring given the
total order. For a total order $\l$ on an interval $[a,b]$, define functions
$L=L_\l$ and $R=R_\l$ on $[a,b]$ by
\begin{equation}\label{local-nbrs}
\begin{aligned}
L(i)&:=\max\{j<i: j\l i\}; \\
R(i)&:=\min\{j>i: j\l i\},
\end{aligned}
\end{equation}
where $\min\emptyset=\max\emptyset:=\infty$. Note that the founders of $\l$
are the elements $i$ for which $L(i)$ or $R(i)$ is infinite.

By considering the insertion procedure for the WI model on $[a,b]$, we
seethat the conditional law of the coloring given the order $\L$ can be
expressed as follows.  We assign colors to the elements of $[a,b]$ one by
one, in the order $\L$.  Conditional on the previous choices, the color $X_i$
assigned to $i$ is chosen uniformly at random from the set
$$[q]\setminus\{X_{L(i)},X_{R(i)}\}.$$
(The set has size $q$ at the first step, and subsequently has size $q-1$ if
$i$ is a founder, and otherwise $q-2$.  Of course, $L(i)$ and $R(i)$
correspond to the neighbors of $i$ when it was inserted.)

The above sequential coloring may be done in other orders.  Specifically,
consider the directed graph $G=G(\L)$ with vertex set $[a,b]$ and with edges
from $i$ to $L(i)$, and from $i$ to $R(i)$, for each $i$, wherever these
quantities are finite.  There is a {\em partial} order $\prec$ on $[a,b]$
generated by the set of inequalities $\{i\succ j: G \text{ has an edge from
}i \text{ to }j\}$.  Then the above sequential coloring procedure may be done
in any order that is a linear extension of $\prec$.  The resulting
conditional law is the same for all such linear extensions.

The graph $G$, and a coloring, are illustrated in \cref{cactus}.  It is
helpful to draw vertex $i$ with horizontal coordinate $i$ and vertical
coordinate given by its position in $\L$, as before.  Then $G$ has a
``lower'' path composed of all the founders, with edges directed inwards
along the path towards the earliest element in the order.  On each edge of
this path, there is a structure built of triangles, each with its base on an
existing edge (starting with the edge of the path) and its apex above that
edge.
\begin{figure}\centering
  \includegraphics[width=0.65\textwidth]{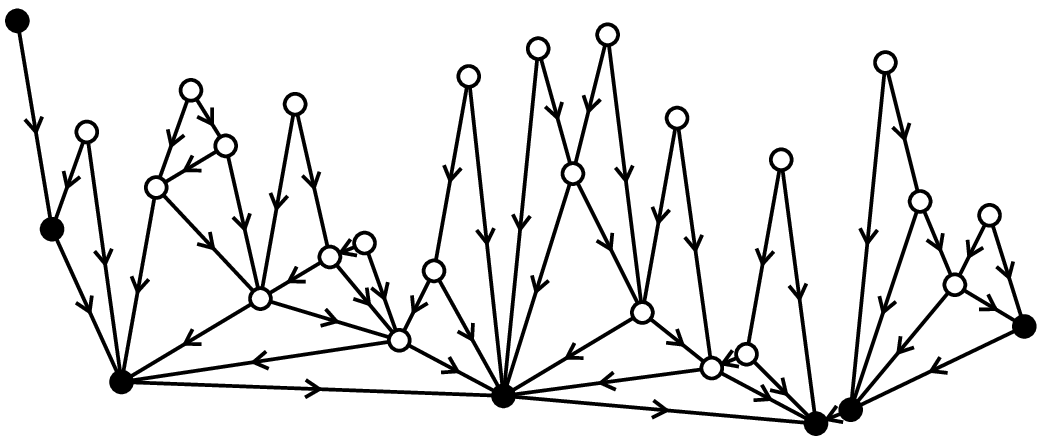}\\
  \includegraphics[width=0.65\textwidth]{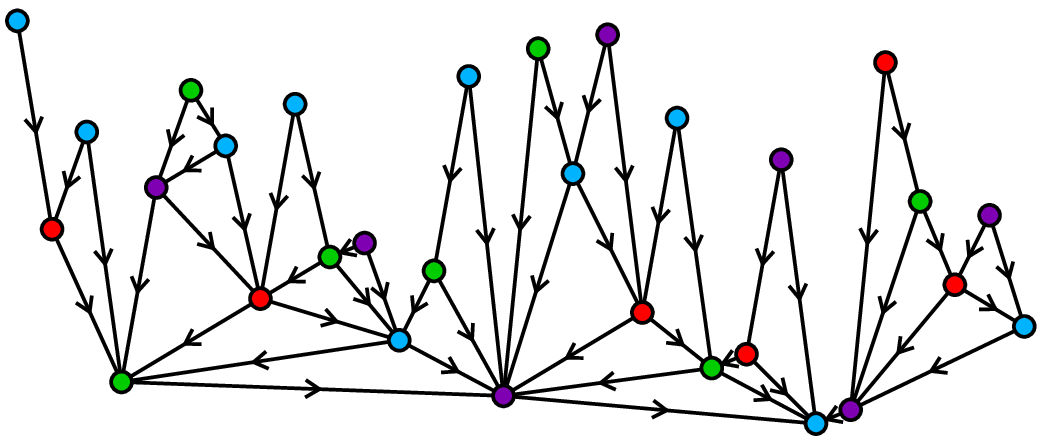}\\
  \includegraphics[width=0.65\textwidth]{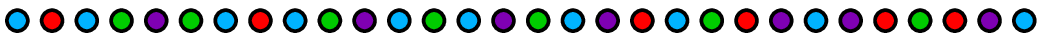}
\caption{{\em Top:} the directed graph $G$ corresponding to a random
 total order on the interval $[1,30]$.  Founders are shown as filled discs.
   {\em Middle:} A $4$-coloring of $G$.  {\em Bottom:}
   The resulting $4$-coloring of the interval. }\label{cactus}
\end{figure}

The conditional law has the following Markovian property.

\begin{lemma}\label{eqd-2}
Fix $q\geq 3$, and consider integer intervals $[a,b]\subseteq[A,B]$.  Let
$\l,\l\,'$ be deterministic total orders on $[A,B]$, and let $X,X'$ be random
colorings on $[A,B]$ chosen according to the respective conditional laws
given $\l,\l\,'$ under the WI model with $q$ colors. Suppose that:
\begin{ilist}
  \item
   $\l|_{[a,b]}=\l\,'|_{[a,b]}$, and
  \item
   $\mathcal{F}(\l|_{[a,b]})=\{a,b\}$.
\end{ilist}
Then $X|_{[a,b]}\eqd X'|_{[a,b]}$.
\end{lemma}

\begin{proof}The proof is illustrated in \cref{compare}.
Let $G,G'$ be the directed graphs on $[A,B]$ corresponding to $\l,\l\,'$.
Condition (ii) implies that $G$ has an edge between $a$ and $b$ (in some
direction), and that for all $i\in (a,b)$ we have $i\,\g a$ and $i\,\g b$.
Therefore, the subgraph of $G$ induced by $[a,b]$ includes all the edges of
$G$ that are incident to $(a,b)$, and no edges out of $\{a,b\}$ except the
edge connecting them. By symmetry under permutations of $[q]$, the joint law
$(X_a,X_b)$ must be uniform on the set of $q^2-q$ ordered pairs of unequal
colors.  By (ii), all the same observations apply to $G'$, and the subgraphs
of $G,G'$ induced by $[a,b]$ are identical.  By the sequential coloring
procedure described above, the conditional law of $X|_{(a,b)}$ given
$(X_a,X_b)$ is thus identical to its counterpart for $X'$, concluding the
proof.
\end{proof}
\begin{figure}\centering
  \includegraphics[width=0.65\textwidth]{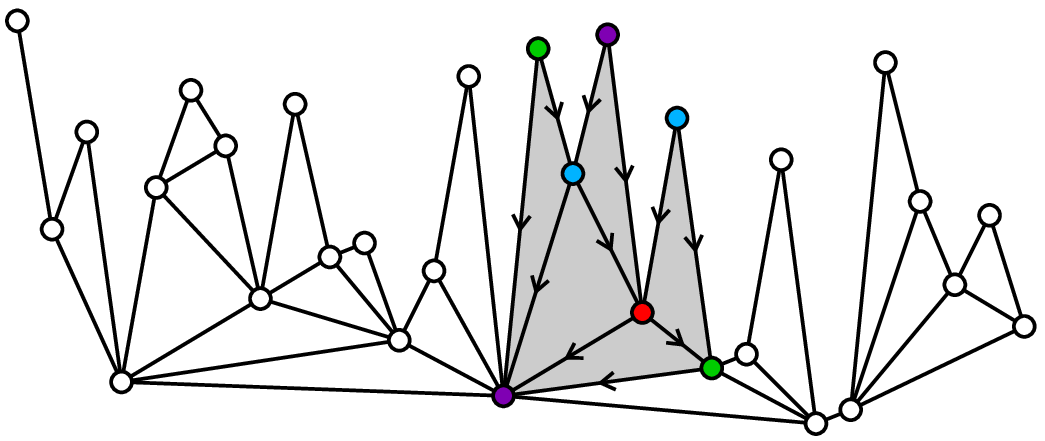}\\
  \includegraphics[width=0.65\textwidth]{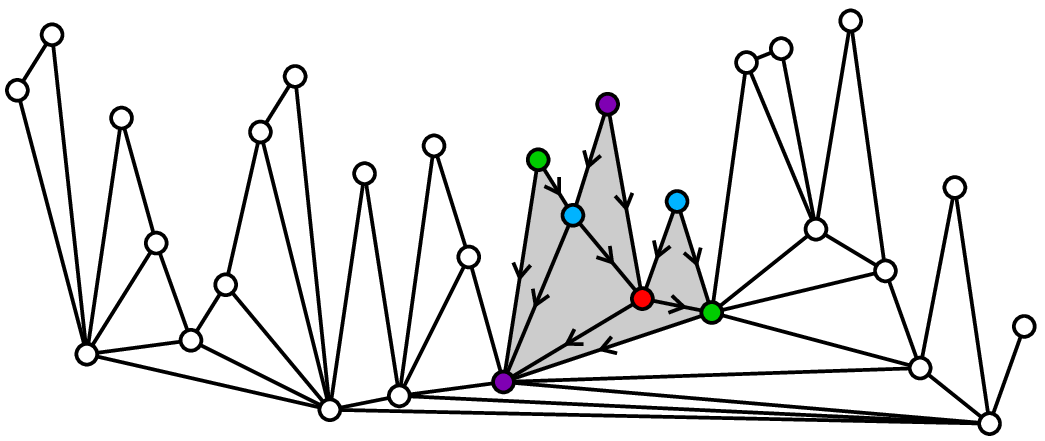}
\caption{An illustration of \cref{eqd-2} and its proof.
The directed graphs associated with two different orders on the interval are shown.
The restrictions to the highlighted subinterval agree, and the only founders of the
restricted order are the endpoints.  Therefore the conditional laws of the colorings
agree on the subinterval.}\label{compare}
\end{figure}

\begin{cor}\label{eqd-3}
Fix $q\geq 3$, and consider integer intervals
$[a,b]\subseteq[A,B]\subseteq[\mathcal{A},\mathcal{B}]$.  Let $\l,\l\,'$ be
deterministic total orders on $[\mathcal{A},\mathcal{B}]$, and let $X,X'$
be random colorings chosen according to their
respective conditional laws.  Suppose that:
\begin{ilist}
  \item
   $\l|_{[A,B]}=\l\,'|_{[A,B]}$, and
  \item
   $\mathcal{F}(\l|_{[A,B]})\cap[a,b]=\emptyset$.
\end{ilist}
Then $X|_{[a,b]}\eqd X'|_{[a,b]}$.
\end{cor}

\begin{proof}  Write $\l\,_0:=\l|_{[A,B]}$.
Define $[c,d]$ to be the minimal interval between founders of $\l\,_0$ that
contains $[a,b]$.  I.e.\ let $c:=\max\{\mathcal{F}(\l\,_0)\cap[A,a-1]\}$ and
$d:=\min\{\mathcal{F}(\l\,_0)\cap[b+1,B]\}$ (which must be finite because the
endpoints $A,B$ are founders of any order on $[A,B]$).  Then
$\mathcal{F}(\l\,_0)\cap[c,d]=\{c,d\}$, which implies that $i\g_0 c$ and
$i\g_0 d$ for all $i\in(c,d)$, and thus $\mathcal{F}(\l|_{[c,d]})=\{c,d\}$.
Now we can apply \cref{eqd-2} to obtain $X|_{[c,d]}\eqd X'|_{[c,d]}$, whence
$X|_{[a,b]}\eqd X'|_{[a,b]}$.
\end{proof}

We can now prove the main result of this section.

\begin{proof}[Proof of \cref{{limit-col}}]
Let $\epsilon>0$.  By \cref{no-found} with weight $1$, choose $M=M(m)>m$
sufficiently large that a uniformly random order on $[-M,M]$ has no founders
in $[-m,m]$ with probability at least $1-\epsilon$.  Now, by
\cref{limit-unif}, choose $n=n(M)>M$ sufficiently large that for both set of
parameters $(q,w)$ and $(q,w^*)$, the restriction of the WI order on $[-n,n]$
to $[-M,M]$ differs from the uniform total order on $[-M,M]$ by at most
$\epsilon$ in total variation.

Let $(X,\L)$ and $(X^*,\L^*)$ be WI models on $[-n,n]$ with respective
parameters $(q,w)$ and $(q,w^*)$.  We will couple them in such a way that
their colorings agree on $[-m,m]$ with high probability.  First, by the
choice of $n$, couple $\L$ and $\L^*$ so that, conditional on some event $E$
of probability at least $1-\epsilon$, we have that
$\L|_{[-M,M]}=\L^*|_{[-M,M]}$, and this restriction is conditionally
uniformly random.  By the choice of $M$, on some further event $E'\subseteq
E$ with $\P(E'\mid E)\geq 1-\epsilon$ (and thus $\P(E')\geq (1-\epsilon)^2$),
the restriction $\L|_{[-M,M]}$ has no founders in $[-m,m]$.  Now, by
\cref{eqd-3}, we can couple $X$ and $X^*$ (with the correct conditional laws
given $\L$ and $\L^*$) so that on $E'$ we have $X|_{[-m,m]}\eqd
X^*|_{[-m,m]}$.

However, the consistency property in \cref{col-cons} implies that
$X^*|_{[-m,m]}$ is equal in law to the WI coloring with parameters $(q,w^*)$
on $[-m,m]$, for all $n$.
\end{proof}

\section{Proof of main result}

\begin{proof}[Proof of \cref{main}]
Let $q=4$, let $(U_i,\phi_i)_{i\in \Z}$ be as in \cref{finitary}, and let
$X=(X_i)_{i\in\Z}$ be the solution to \eqref{system}. By \cref{finitary}, it
suffices to show that $(X_i)_{i\in\Z}$ is equal in law to the $1$-dependent
$4$-coloring constructed in \cite{hl}. By \cref{col-cons,kdep}, this will
follow if we can show that its restriction to the interval $[-m,m]$ is equal
in law to the WI model with weight $w^*(4)=3/2$, for every $m$.

Let $\L$ be the total order on $\Z$ induced by $(U_i)_{i\in\Z}$; i.e.\ let
$i\L j$ if and only if $U(i)<U(j)$.  Let $n\geq 1$ be an integer, let $\L^n$
be the restriction $\L|_{[-n,n]}$, and let define the neighbor maps
$L^n=L_{\L^n}$ and $R^n=R_{\L^n}$ via \eqref{local-nbrs}, so that $L^n(i)$ or
$R^n(i)$ is infinite when $i$ is a founder of $\L^n$.  Now define a coloring
$X^n=(X_{-n},\ldots,X_n)$ to be the solution of \eqref{system}, except
restricted to $i\in[-n,n]$, and using $L^n,R^n$ in place of $L,R$.  We take
$X_\infty:=\infty$, so that when $L^n(i)$ or $R^n(i)$ is infinite, the
restriction involving $X_{L^n(i)}$ or $X_{R^n(i)}$ in \eqref{system} is
ignored. Existence and uniqueness of the solution is clear by considering the
integers $i\in[-m,m]$ in the order $\L$.

Since the preference orders $\phi_i$ are uniformly random, $X^n$ is equal in
law to the WI coloring on $[-n,n]$ with $4$ colors and weight $w=1$.  On the
other hand, let $n>m$, and let $G_n$ be the event that there exist lucky integers
$A,B$ with $[-m,m]\subseteq[A,B]\subseteq [-n,n]$ and $i\G A,B$ for all $i\in
(A,B)$. Then by the argument in the proof of \cref{finitary}, $X^n|_{[-m,m]}$
and $X|_{[-m,m]}$ are equal on $G_n$ (where $X=(X_i)_{i\in\Z}$ is the global
solution to \eqref{system} mentioned earlier).  \cref{lucky} implies that
a.s.\ $G_n$ occurs eventually as $n\to\infty$. Therefore $X^n|_{[-m,m]}\to
X|_{[-m,m]}$ a.s.  But \cref{limit-col} states that $X^n|_{[-m,m]}$ converges
in law to the WI coloring $Y$ with weight $w^*$ on $[-m,m]$, so
$X|_{[-m,m]}\eqd Y$ as required.
\end{proof}

\bibliographystyle{abbrv}
\bibliography{fin}

\end{document}